\documentclass[11.5pt]{amsart}
\usepackage{geometry}
\usepackage{graphicx}
\usepackage[colorlinks=true, citecolor=blue, linkcolor=blue, urlcolor=blue]{hyperref}
\usepackage{nicefrac}
\newgeometry{asymmetric, centering, margin=1in}
\numberwithin{equation}{section}

\newtheorem{lemma}{Lemma}

\newtheorem{theorem}{Theorem}

\newtheorem{definition}{Definition}
\theoremstyle{remark}

\def\R{\mathbb R}

\def\Z{\mathbb Z}
\def\N{\mathbb N}
\def\M{\mathcal M}

\def\d{\partial}
\newcommand{\pair}[1]{\left\langle #1 \right\rangle}
\newcommand{\norm}[1]{\left\lVert #1 \right\rVert}
\newcommand{\abs}[1]{\left\lvert #1 \right\rvert}

\date{Compiled \today}
\title[H\"older Stable Recovery of Time-Dependent Electromagnetic Potentials]{H\"older Stable Recovery of Time-Dependent Electromagnetic Potentials Appearing in a Dynamical Anisotropic Schr\"odinger Equation}
\author[Y. Kian]{Yavar Kian}
\author[A. Tetlow]{Alexander Tetlow}

\begin{document}
\maketitle

\begin{abstract}
We consider the inverse problem of H\"oldder-stably determining the time- and space-dependent coefficients of the Schr\"odinger equation on a simple Riemannian manifold with boundary of dimension $n\geq2$ from knowledge of the Dirichlet-to-Neumann map. Assuming the divergence of the magnetic potential is known, we show that the electric and magnetic potentials can be H\"older-stably recovered from these data.  Here we also remove the  smallness assumption for the solenoidal part of the magnetic potential present in previous results.
\end{abstract}

\section{Introduction}

\subsection{Statement of the Problem} Let $T>0$, let $(\M,g)$ be a compact, connected, smooth Riemannian manifold of dimension $n\geq2$, and denote by $\d\M$ its boundary. Further assume that $(\M,g)$ is simple (see definition \ref{d1}). Let $A\in W^{2,\infty}((0,T)\times\M;T^\ast\M)$ be given by $A=\sum_{j=1}^na_jdx^j$, and consider the magnetic Laplacian given by\[\Delta_{g,A(t)}u=\sum_{j,k=1}^{n}\abs{g}^{-\frac{1}{2}}\big(\d_{x^j}+ia_j(t,x)\big)\Big(\abs{g}^{\frac{1}{2}}g^{jk}\big(\d_{x^k}+ia_k(t,x)\big)u\Big),\] where $g^{-1}=g^{ij}$ and $\abs{g}=\det(g)$. If $A=0$, this is just the usual Laplace-Beltrami operator $\Delta_g$. For $T>0$ and $q\in W^{1,\infty}((0,T)\times\M)$ we consider the initial boundary value problem (IBVP)\begin{equation}\begin{split}\label{1.1}i\d_tu(t,x)+\Delta_{g,A(t)}u(t,x)+q(t,x)u(t,x)&=0\textrm{ in }(0,T)\times\M,\\u(t,x)&=f\textrm{ on }(0,T)\times\d\M,\\u(0,x)&=0\textrm{ in }\M,\end{split}\end{equation} with inhomogeneous Dirichlet data $f$. For all $r,s\in(0,\infty)$ and $X=\M$ or $X=\d\M$ define the spaces $H^{r,s}((0,T)\times X)=H^r(0,T;L^2(X))\cap L^2(0,T;H^s(X))$ with the associated norm\[\norm{u}_{H^{r,s}((0,T)\times X)}^2=\norm{u}_{H^r(0,T;L^2(X))}^2+\norm{u}_{L^2(0,T;H^s(X))}^2.\]
We further define the space\[H^{r,s}_0((0,T)\times\d\M)=\Big\{f\in H^{r,s}\big((0,T)\times\d\M\big):\textrm{ for all }k\in\big(-1,s-\frac{1}{2}\big)\cap\N,\ \d^k_tf\vert_{t=0}=0\Big\}.\]

The problem (\ref{1.1}) admits a unique solution $u\in H^{1,2}((0,T)\times\M)$ for $f\in H^{\frac{9}{4},\frac{3}{2}}((0,T)\times\d\M)$ (see \cite[Proposition 2.1]{22}). Further, the Dirichlet-to-Neumann (D-to-N map in short) map\begin{equation}\label{1.2}\Lambda_{A,q}(f)=(\d_\nu+iA\nu)u,\textrm{ for }f\in H^{\frac{9}{4},\frac{3}{2}}((0,T)\times\d\M),\end{equation}where $\nu=\nu(x)$ denotes the unit outward normal to $\d\M$ with respect to the metric $g$, is a bounded operator from $H^{\frac{9}{4},\frac{3}{2}}_0((0,T)\times\d\M)$ to $L^2((0,T)\times\d\M)$. For $j=1,2$, let $A_j\in W^{2,\infty}((0,T)\times\M;T\ast\M)$, and $q_j\in W^{1,\infty}((0,T)\times\M)$. We call $(A_1,q_1)$ and $(A_2,q_2)$ gauge equivalent if there exists $\phi\in W^{3,\infty}((0,T)\times\M)$ such that $\phi\vert_{(0,T)\times\d\M}=0$, $A_2=A_1+d\phi$ and $q_2=q_1-\d_t\phi$ and let $u_j$ be the solution of (\ref{1.1}) with potentials $A=A_j$ and $q=q_j$. If $\phi$ is as above, we recall that the D-to-N map is invariant under this gauge transformation. More precisely, we have\[(i\d_t+\Delta_{g,A_1(t)}+q_1)e^{i\phi}u_2(x,t)=e^{i\phi}(i\d_t+\Delta_{g,A_2(t)}+q_2)u_2(x,t)=0,\]and we deduce that $e^{i\phi}u_2=u_1$ and\[(\d_\nu+iA_1\nu)u_1=(\d_\nu+i(A_1+d\phi)\nu)u_2=(\d_\nu+iA_2\nu)u_2,\]which then implies that $\Lambda_{A_1,q_1}=\Lambda_{A_2,q_2}$. This obstruction to uniqueness notwithstanding, the aim of this paper is to prove H\"older-stable recovery of the time-dependent electric and magnetic potentials $(A,q)$ from knowledge of the D-to-N map $\Lambda_{A,q}$.

\subsection{History of the Problem}In the case of the dynamic Schr\"odinger equation with time-independent potentials, H\"older-stable recovery of the magnetic field from knowledge of the Dirichlet-to-Neumann map was shown in \cite{3}, and stable recovery of the electric potential of the Schr\"odinger equation on a Riemannian manifold was proved in \cite{4}. This latter result is extended to stable determination of the electromagnetic potentials on a Riemannian manifold from the D-to-N map in \cite{2}. We mention also the recent work of \cite{BKS}, where such results have been extended to unbounded cylindrical domain.

Literature dealing with the inverse problem of recovering time-dependent potentials of the Schr\"odinger equation is rather sparse. To the best of the authors’ knowledge, the only results establishing recovery of time-dependent potentials of the Schr\"odinger equation from the D-to-N map deal with Euclidean domains. In particular, it was proved in \cite{11} that the time-dependent electric and magnetic potentials are uniquely determined by the D-to-N map. Logarithmic-stable determination was shown for the electric potential in \cite{9}. This result was extended to the full electromagnetic potential  in \cite{7}, provided that the time-independent part of the magnetic potential is sufficiently small. Indeed, it was only recently shown in \cite{22} that the electromagnetic potential in a Euclidean domain can be H\"older-stably recovered from knowledge of the D-to-N map.

In the current work, we show that it is possible to H\"older-stably recover the time-and-space-dependent coefficients of the dynamic Schr\"odinger equation on a simple Riemannian manifold.

\subsection{Main Results} Here and in the rest of this paper we write $\norm{\cdot}$ for the norm of an operator in $\mathcal{B}\big(H^{\frac{9}{4},\frac{3}{2}}_0((0,T)\times\d\M),L^2((0,T)\times\d\M)\big)$. In this paper we aim to prove the following:

\begin{theorem}\label{T1}
(Uniqueness):For $j=1,2$, let $A_j\in W^{6,\infty}((0,T)\times\M;T^\ast\M)$ and $q_j\in W^{4,\infty}((0,T)\times\M)$. Assume also that\begin{equation}\label{1.3}\d^\alpha_xA_1(t,x)=\d^\alpha_xA_2(t,x),\quad(t,x)\in(0,T)\times\d\M,\quad\alpha\in\N^n,\abs{\alpha}\leq5.\end{equation}
Then the condition $\Lambda_{A_1,q_1}=\Lambda_{A_2,q_2}$ implies that $(A_1,q_1)$ and $(A_2,q_2)$ are gauge equivalent.
\end{theorem}

\begin{theorem}\label{T2}
(Stable Recovery of the Magnetic Potential): Let the condition of Theorem \ref{T1} be fulfilled and, for $j=1,2$,  let $A_j\in W^{6,\infty}((0,T)\times\M;T^\ast\M)\cap H^{3n+4}((0,T)\times\M;T^\ast\M)$ be such that 
\begin{equation}\label{1.333}\d^\alpha_xA_1(t,x)=\d^\alpha_xA_2(t,x),\quad(t,x)\in(0,T)\times\d\M,\quad\alpha\in\N^n,\abs{\alpha}\leq 3n+3.\end{equation}
Assume also that there exists a constant $B$ such that \begin{equation}\label{C1}\sum_{j=1,2}\norm{q_j}_{W^{5,\infty}((0,T)\times\M;T^\ast\M)}+\norm{A_j}_{W^{5,\infty}((0,T)\times\M;T^\ast\M)}+\norm{A_j}_{H^{3n+4}((0,T)\times\M;T^\ast\M)}\leq B.\end{equation}

Then we have \[\norm{A_1^{sol}-A_2^{sol}}\leq C\norm{\Lambda_{A_1,q_1}-\Lambda_{A_2,q_2}}^{s_1},\]where $s_1>0$ is a general  constant, $C>0$ a constant depending only on $B$, $T$, $\M$ and $A_j^{sol}$ is the solenoidal part of the Hodge decomposition of $A_j$, given in Lemma \ref{L1}.
\end{theorem}

\begin{theorem}\label{T3}
(Stable Recovery of the Electric Potential): Let the condition of Theorem \ref{T2} be fulfilled with 
\begin{equation}\label{t1.6}\delta A_1=\delta A_2.\end{equation}
Fix also $q_j\in W^{4,\infty}((0,T)\times\M)\cap H^5((0,T)\times\M)$ and assume that the condition \begin{equation}\label{1.4}\d^\alpha_xq_1(t,x)=\d^\alpha_xq_2(t,x),\quad(t,x)\in(0,T)\times\d\M,\quad\alpha\in\N^n,\ \abs{\alpha}\leq4,\end{equation} is fulfilled. We also assume that there exists a constant $B_1>0$ such that\begin{equation}\label{1.5}\sum_{j=1,2}\big(\norm{q_j}_{W^{4,\infty}((0,T)\times\M)}+\norm{q_j}_{H^5((0,T)\times\M)}\big)\leq B_1.\end{equation}

Then we have \begin{equation}\label{1.6}\norm{q_1-q_2}_{L^2((0,T)\times\M)}\leq C\norm{\Lambda_{A_1,q_1}-\Lambda_{A_2,q_2}}^{s_2},\end{equation} where $C$ depends only on $B$, $B_1$ $T$, and $\M$, and $s_2$ is a general constant.
\end{theorem}

As far as the authors are aware, the present work is the first dealing with recovery of time-dependent potentials appearing in a Schr\"odinger equation with variable coefficients of order two. In fact, the above estimates are the first showing H\"older-stable  recovery of a coefficient dependent on all variables of a second order partial differential equation with variable coefficients of order two. The only other work where similar results have been obtained is \cite{22}, where the authors consider the case of a bounded subset of $\R^n$ with the Euclidean metric.

Furthermore, stable recovery of a magnetic potential appearing in a Schr\"odinger equation on a manifold with non-Euclidean metric has, thus far, relied upon the a priori assumption that the magnetic potential is small in some appropriate norm, even in the time-independent case (see, for example, \cite{2}). This smallness assumption is also utilized when recovering the magnetic potential of the wave equation (as seen in \cite{34}). In fact, it happens that this assumption is not necessary when dealing with the Schr\"odinger equation, even when the magnetic potential is allowed to depend on time, as we shall demonstrate herein.

In Section \ref{section2}, we introduce the geodesic ray-transforms for $1$-forms and for functions. In Section \ref{section3} we construct geometric optics solutions to the equation (\ref{1.1}). We devote Section \ref{section4} to the proof of Theorem \ref{T1}, using the geometric optics solutions as the main tool. The estimate of Theorem \ref{T2} is proved in Section \ref{section5}, whereas the estimate of Theorem \ref{T3} is proved in Section \ref{section6}.

\section{Notations}\label{section2}
In this section, we list some notation used in the rest of the paper. We denote by $\pair{\cdot,\cdot}_g$ the inner product with respect to $g$ on $T\M$, that is for $x\in\M$ and $Y,Z\in T_x\M$ given by $Y=\sum_{j=1}^ny_j\d_{x^j}$, $Z=\sum_{j=1}^nz_j\d_{x^j}$ we have\[\pair{Y,Z}_{g(x)}=\sum_{j,k=1}^ng_{jk}(x)y_jz_k.\]

Similarly, we denote by $\pair{\cdot,\cdot}_g$ the inner product with respect to $g$ on $T^\ast\M$, that is for $U,V\in T_x^\ast\M$ given by $U=\sum_{j=1}^nu_jdx^j$, $V=\sum_{j=1}^nv_jdx^j$ we have\[\pair{U,V}_g(x)=\sum_{j,k=1}^ng^{jk}(x)u_jv_k.\]

We denote by $dV_g$ the Riemannian volume on $\M$, which is given in local coordinates by ${dV_g=\abs{g}^{\frac{1}{2}}dx^1\wedge\cdots\wedge dx^n}$. We further define on $\d\M$ the surface measure $\sigma_g$ such that for $X\in H^1(\M;T\M)$ we have \[\int_\M\textrm{div}_g(X)dV_g=\int_{\d\M}\pair{X,\nu}_gd\sigma_g,\]where $\textrm{div}_g(X)=\sum_{j=1}^n\abs{g}^{-\frac{1}{2}}\d_{x_j}\big(\abs{g}^{\frac{1}{2}}X^j\big)$. Additionally, we recall the Riemannian gradient operator given by $\nabla_gf=\big(g^{j1}\d_{x_j}f,\cdots,g^{jn}\d_{x_j}f\big)$.

We recall the coderivative operator $\delta$ is the operator sending the 1-form $\omega=\sum_{i=1}^n\omega_idx^i\in W^{1,\infty}(\M;T^\ast\M)$ to the function $\delta\omega$ given in local coordinates by\begin{equation}\label{coderivative}\delta\omega=\abs{g}^{-\frac{1}{2}}\sum_{j,k=1}^n\d_{x^j}\big(\abs{g}^{\frac{1}{2}}g^{jk}\omega_k\big).\end{equation}

We recall also the definition of a simple manifold. Let $D$ be the Levi-Civita connection on $(\M,g)$. For $x\in\d\M$ we consider the second quadratic form of the boundary\[\Pi(\theta,\theta)=\pair{D_\theta\nu,\theta}_{g(x)},\ \theta\in T_x\d\M.\]

We say that $\d\M$ is strictly convex if the form $\Pi$ is positive-definite for every $x\in\d\M$.

\begin{definition}\label{d1}We say that $(\M,g)$ is simple if $\d\M$ is strictly convex, $\M$ is simply connected, and for any $x\in\M$ the exponential map $\exp_x:\exp_x^{-1}(\M)\rightarrow\M$ is a diffeomorphism.\end{definition}

We write $\gamma_{x,\theta}$ for the unique geodesic in $\M$ with initial point $x\in\M$ and initial direction $\theta\in T_x\M$. We define the sphere bundle of $\M$ by\[S\M=\{(x,\theta)\in T\M:\abs{\theta}_g=1\},\] and likewise the submanifold of inner vectors $\d_+S\M$ by\[\d_+S\M=\{(x,\theta)\in S\M,\ x\in\d\M,\ \pair{\theta,\nu(x)}_g(x)<0\}.\]

Given that $\M$ is assumed to be simple, we can also define $\tau_+(x,\theta)$ to be the maximal time of existence in $M$ of the geodesic $\gamma_{x,\theta}$ for $x\in\d\M$, that is\[\tau_+(x,\theta)=\min\{s>0:\gamma_{x,\theta}(s)\in\d\M\}\textrm{ for }(x,\theta\in\d_+S\M).\]

We also introduce here the geodesic ray transforms on a simple Riemannian manifold $\M$.
\begin{definition}\label{d2}The geodesic ray transform for $1$-forms is the linear operator $I_1:\mathcal{C}^\infty(\M;T^\ast\M)\rightarrow\mathcal{C}^\infty(\d_+S\M)$ which is defined by\begin{equation}\label{I1}I_1\omega(x,\theta)=\int_0^{\tau_+(x,\theta)}\omega(\gamma_{x,\theta}(s))\gamma'_{x,\theta}(s)ds,\quad(x,\theta)\in\d_+S\M,\ \omega\in\mathcal{C}^\infty(\M;T^\ast\M).\end{equation}\end{definition}

\begin{definition}\label{d3}The geodesic ray transform for functions is the linear operator $I_0:\mathcal{C}^\infty(\M)\rightarrow\mathcal{C}^\infty(\d_+S\M)$ which is given by\begin{equation}\label{I0}I_0f(x,\theta)=\int_0^{\tau_+(x,\theta)}f(\gamma_{x,\theta}(s))ds,\quad(x,\theta)\in\d_+S\M,\ f\in\mathcal{C}^\infty(\M).\end{equation}\end{definition}

\section{Geometric Optics Solutions}\label{section3}
We now seek to construct GO solutions of the magnetic Schr\"odinger equation in $(0,T)\times\M$. We fix $A_j\in W^{6,\infty}((0,T)\times\M;T^\ast\M)$, $q_j\in W^{4,\infty}((0,T)\times\M)$ and assume that\begin{equation}\label{2.1}\d_x^\alpha A_1(t,x)=\d_x^\alpha A_2(t,x),\quad(t,x)\in(0,T)\times\d\M,\quad\alpha\in\N^n,\ \abs{\alpha}\leq5.\end{equation}

We consider the equations
\begin{equation}\begin{split}\label{2.2}i\d_tu_j+\Delta_{g,A_j(t)}u_j+q_ju_j&=0\textrm{ in }(0,T)\times\M,\\u_1(0,\cdot)=u_2(T,\cdot)&=0\textrm{ in }\M.\end{split}\end{equation}

We seek to find, for $\lambda>1$, $j=1,2$, solutions $u_j\in H^{1,2}((0,T)\times\M)$ of (\ref{2.2}) of the form\begin{equation}\label{2.4}u_j(t,x)=\Big(a_j(t,x)+\frac{b_j(t,x)}{\lambda}\Big)e^{i\lambda(\psi(x)-\lambda t)}+R_{j,\lambda}(t,x).\end{equation}

In (\ref{2.4}) above, $\psi,a_j,b_j$ satisfy the following eikonal and transport equations:
\begin{equation}\label{2.5}\abs{\nabla_g\psi}_g^2=1,\end{equation}
\begin{equation}\label{2.6}2i\pair{\nabla_g\psi,\nabla_ga_j}_g+i(\Delta_g\psi)a_j-2(A_j\nabla_g\psi)a_j=0,\end{equation}
\begin{equation}\label{2.7}2i\pair{\nabla_g\psi,\nabla_ga_j}_g+i(\Delta_g\psi)b_j-2(A_j\nabla_g\psi)b_j=-(i\d_t+\Delta_{g,A_j(t)}+q_j)a_j.\end{equation}

Taken together, equations (\ref{2.5}) - (\ref{2.7}) yield\[(i\d_t+\Delta_{g,A(t)}+q_j)\Big[e^{i\lambda(\psi(x)-\lambda t)}\Big(a_j(t,x)+\frac{b_j(t,x)}{\lambda}\Big)\Big]=e^{i\lambda(\psi(x)-\lambda t)}\frac{(i\d_t+\Delta_{g,A(t)}+q_j)b_j(t,x)}{\lambda}.\]

We also assume that there exists $\tau\in\big(0,\frac{T}{4})$ such that $a_j,b_j$ are supported in $[\tau,T-\tau]\times\M$ and further assume that $a_j,b_j\in H^3((0,T)\times\M)$, whence $(i\d_t+\Delta_{g,A_j(t)}+q_j)b_j\in H^1(0,T;L^2(\M))$. Thus we can choose $R_{j,\lambda}$ solving\begin{align}\begin{aligned}\label{2.9}(i\d_t+\Delta_{g,A_j(t)}+q_j)R_{j,\lambda}&=-e^{i\lambda(\psi(x)-\lambda t)}\frac{(i\d_t+\Delta_{g,A_j(t)}+q_j)b_j}{\lambda}&&\textrm{ in }(0,T)\times\M,\\R_{1,\lambda}(0,\cdot)&=R_{2,\lambda}(T,\cdot)=0&&\textrm{ in }\M,\\R_{j,\lambda}(t,x)&=0&&\textrm{ on }(0,T)\times\d\M.\end{aligned}\end{align}

Since $(\M,g)$ is simple, the eikonal equation (\ref{2.5}) can be solved globally on $\M$. To see this, we first extend the simple manifold $(\M,g)$ to a simple, compact manifold $(\M_1,g)$  with $\M$ contained in the interior of $\M_1$. We pick $y\in\d\M_1$ and consider polar normal coordinates $(r,\theta)$ on $\M_1$ given by $x=\exp_y(r\theta)$ for $r>0$ and $\theta\in S_y\M_1=\{v\in T_y\M_1:\abs{v}_{g(y)}=1\}$. Letting $\nu(y)$ denote the outward unit normal to $\d\M_1$ with respect to the metric $g$, we define $\d_+S_y\M_1=\{\theta\in S_y\M_1:\pair{\theta,\nu(y)}_{g(y)}<0\}$. According to the Gauss Lemma (see e.g. \cite[Chapter 9, Lemma 15]{29}), in these coordinates the metric takes the form $g(r,\theta)=dr^2+g_0(r,\theta)$ with $g_0(r,\theta)$ a metric on $\{\theta\in S_y\M_1:\pair{\nu(y),\theta}_{g(y)}\leq0\}$ depending smoothly on $r$. In polar normal coordinates $dV_{g}=\mu(r,\theta)^{\frac{1}{2}}drd\theta$, where $\mu=\det{g_0}$ and $d\theta$ is the usual spherical volume form on $\d_+S_y\M$. For a function $f\in L^1(\M)$ extended by zero to $\M_1$, we can extend $dV_{g}$ to a volume form on $T_y(\M_1)$ and get\[\int_\M f(x)dV_g(x)=\int_0^\infty\int_{\d_+S_y\M_1}f(r,\theta)\mu(r,\theta)^{\frac{1}{2}}drd\theta.\]

We choose \begin{equation}\label{2.11}\psi(x)=dist_g(y,x)\end{equation}where $dist_g$ denotes the Riemannian distance function. Since $\psi(r,\theta)=r$, we can easily check that $\psi$ solves the eikonal equation (\ref{2.5}).\\

We now look towards solving the transport equations (\ref{2.6})-(\ref{2.7}). First, note that \begin{equation}\label{2c1}\nabla_g\psi(r,\theta)=\d_r=\gamma'_{y,\theta}(r)=\theta.\end{equation}

Therefore, we rewrite the transport equations (\ref{2.6})-(\ref{2.7}) in polar normal coordinates based at $y\in\d\M_1$ to obtain\begin{equation}\label{tp1}\d_ra_j+\Big(\frac{\d_r\mu}{4\mu}\Big)a_j+i\big(A_j\theta\big)a_j=0,\end{equation}\begin{equation}\label{tp2}\d_rb_j+\Big(\frac{\d_r\mu}{4\mu}\Big)b_j+i(A_j\theta)b_j=\beta_j(t,r,\theta),\end{equation}where $A_j\theta$ denotes $A_j(t,r,\theta)\theta$ and $\beta_j$ denotes $(i\d_t+\Delta_{g,A_j(t)}+q_j)a_j/2$.\\

Applying \cite[Section 3, Theorem 5]{33}, we find $\tilde{A}_1\in W^{6,\infty}((0,T)\times\M_1;T^\ast\M_1) $ such that for $t\in(0,T)$ the support of $\tilde{A}_1(t,\cdot)$ is contained in the interior of $\M_1$, and we have $\tilde{A}_1=A_1$ on $(0,T)\times\M$ and $\tilde{\norm{A_1}}_{W^{6,\infty}((0,T)\times\M_1;T^\ast\M_1)}\leq C\norm{A_1}_{W^{6,\infty}((0,T)\times\M;T^\ast\M)}$, where $C$ depends only on $\M$. Then for all $t\in(0,T)$ we put:\[\tilde{A}_2(t,x)=\begin{cases}A_2(t,x),\textrm{ if }x\in\M,\\\tilde{A}_1(t,x),\textrm{ if }x\in\M_1\setminus\M.\end{cases}\]
Then according to (\ref{2.1}), $\tilde{A}_2\in W^{6,\infty}((0,T)\times\M_1;T^\ast\M_1)$ and
\[\max_{j=1,2}\norm{\tilde{A}_j}_{ W^{6,\infty}((0,T)\times\M_1;T^\ast\M_1)}\leq C\max_{j=1,2}\norm{A_j}_{ W^{6,\infty}((0,T)\times\M;T^\ast\M)}.\]
Similarly, for $j=1,2$, we consider $\tilde{q}_j\in W^{4,\infty}((0,T)\times\M_1) $ such that for $t\in(0,T)$ the support of $\tilde{q}_j(t,\cdot)$ is contained in the interior of $\M_1$, and we have $\tilde{q}_j=q_j$ on $(0,T)\times\M$ and $\tilde{\norm{q_j}}_{W^{4,\infty}((0,T)\times\M_1)}\leq C\norm{q_j}_{W^{4,\infty}((0,T)\times\M)}$. Note that here we do not impose that $\tilde{q}_1$ and $\tilde{q}_2$ should coincide on $(0,T)\times (\M_1\setminus \M)$. 

For any $h\in H^5((0,T)\times \d_+S_y\M_1)$, the functions\begin{equation}\label{2.13}a_1(t,r,\theta)=\chi(t)h(t,\theta)\mu(r,\theta)^{-\frac{1}{4}}\exp\Big(i\int_0^{+\infty}\tilde{A}_1(t,r+s,\theta)\theta ds\Big),\end{equation}
\begin{equation}\label{2.14}a_2(t,r,\theta)=\chi(t)\mu(r,\theta)^{-\frac{1}{4}}\exp\Big(i\int_0^{+\infty}\tilde{A}_2(t,r+s,\theta)\theta ds\Big),\end{equation}
are solutions to the transport equations (\ref{tp1}). In the same way, for $\tilde{\beta}_j=(i\d_t+\Delta_{g,\tilde{A}_j(t)}+\tilde{q}_j)a_j/2$, we fix
\begin{equation}\label{2.15}b_j(t,r,\theta)=\mu(r,\theta)^{-\frac{1}{4}}\int_0^r\Big[\exp\Big(-i\int_{s_2}^r\tilde{A}_j(t,s_1,\theta)\theta ds_1\Big)\beta_j(t,s_2,\theta)\mu^{\frac{1}{4}}(s_2,\theta)\Big]ds_2\end{equation}
which is a solution of (\ref{tp2}). Here we fix $\chi\in C_0^\infty((\tau,T-\tau))$ satisfying $\chi=1$ on $[2\tau,T-2\tau]$, $0\leq\chi\leq1$ and $\norm{\chi}_{W^{k,\infty}(\R)}\leq C_k\tau^{-k}$ with $C_k$ independent of $\tau$.\\

Let us now consider the remainder terms $R_{j,\lambda}$, $j=1,2$. In view of (\ref{2.13})-(\ref{2.15}), we deduce the following bounds:
\begin{equation}\label{2.16}\norm{a_1}_{H^3((0,T)\times\M)}\leq C\norm{h}_{H^3((0,T)\times \d_+S_y\M_1)}\tau^{-3},\quad\norm{b_1}_{H^3((0,T)\times\M)}\leq C\norm{h}_{H^5((0,T)\times \d_+S_y\M_1)}\tau^{-4},\end{equation}
\begin{equation}\label{2.17}\norm{(i\d_t+\Delta_{g,A_1(t)}+q_1)b_1}_{L^2((0,T)\times\M)}\leq C\norm{h}_{H^4((0,T)\times \d_+S_y\M_1)}\tau^{-2},\end{equation}
\begin{equation}\label{2.18}\norm{a_2}_{H^3((0,T)\times\M)}\leq C\tau^{-3},\quad\norm{b_2}_{H^3((0,T)\times\M)}\leq C\tau^{-4},\quad\norm{(i\d_t+\Delta_{g,A_2(t)}+q_2)b_2}_{L^2((0,T)\times\M)}\leq C\tau^{-2},\end{equation}
where $C$ depends only on $\M$, $T$ and $\norm{A_1}_{W^{5,\infty}((0,T)\times\M}+\norm{A_2}_{W^{5,\infty}((0,T)\times\M}$. Then applying \cite[Lemma 2.1]{22}, we see that problem (\ref{2.9}) admits unique solutions $R_{j,\lambda}$ for $j=1,2$ with $R_{j,\lambda}\in C([0,T];H^1_0(\M)\cap H^2(\M))\cap\mathcal{C}^1([0,T];L^2(\M))$. On the other hand, from the a priori estimate \cite[(10.10), page 324]{23}, we deduce that\begin{equation}\label{2.19}\norm{R_{1,\lambda}}_{L^2((0,T)\times\M)}\leq C\frac{\norm{(i\d_t+\Delta_{A_1(t)}+q_1)b_1}_{L^2((0,T)\times\M)}}{\lambda}\leq C\norm{h}_{H^4((0,T)\times \d_+S_y\M_1)}\tau^{-2}\lambda^{-1}.\end{equation}
Moreover, applying \cite[Lemma 2.1]{22} we find that\[\norm{R_{1,\lambda}}_{L^2(0,T;H^2(\M))}\leq C\frac{\norm{e^{i\lambda(\psi(x)-\lambda t)}(i\d_t+\Delta_{A_1(t)}+q_1)b_1}_{H^1(0,T;L^2(\M))}}{\lambda}\leq C\norm{h}_{H^4((0,T)\times \d_+S_y\M_1)}\tau^{-3}\lambda,\]and by interpolation between this estimate and (\ref{2.19}) we deduce\[\norm{R_{1,\lambda}}_{L^2(0,T;H^1(\M))}\leq C\norm{h}_{H^4((0,T)\times \d_+S_y\M_1)}\tau^{-3}.\]
Combining this with (\ref{2.19}) we obtain\begin{equation}\label{2.20}\norm{R_{1,\lambda}}_{L^2(0,T;H^1(\M))}+\lambda\norm{R_{1,\lambda}}_{L^2((0,T)\times\M)}\leq C\norm{h}_{H^4((0,T)\times \d_+S_y\M_1)}\tau^{-3}.\end{equation}
In a similar manner, we derive the estimate \begin{equation}\label{2.21}\norm{R_{2,\lambda}}_{L^2(0,T;H^1(\M))}+\lambda\norm{R_{2,\lambda}}_{L^2((0,T)\times\M)}\leq C\tau^{-3}.\end{equation}
This completes our construction of the geometric optics solutions of (\ref{2.2}).

\section{Unique Determination of the Potentials Modulo Gauge Invariance}\label{section4}
 We recall that any $1$-form $\omega\in W^{1,p}(\M;T^\ast\M)$, with $p\in[2,\infty)$ admits a Hodge decomposition via $\omega=\omega^{sol}+d\phi$, where $\omega^{sol}\in W^{1,p}(\M;T^\ast\M)$ is the solenoidal part of $\omega$ which satisfies $\delta\omega^{sol}=0$ (see (\ref{coderivative}) for the definition of coderivative operator $\delta$) and $\phi\in W^{2,p}(\M)\cap H^1_0(\M)$. Let us first prove an extension of this Hodge decomposition for the $1$-form $A\in W^{6,\infty}((0,T)\times\M;T^\ast\M)$ given by the following:

\begin{lemma}\label{L1}
Let $A\in W^{6,\infty}((0,T)\times\M;T^\ast\M)$. Then we can decompose $A$ into\begin{equation}\label{3.1}A=A^{sol}+d\phi,\end{equation}where, for any $p\in(2,\infty)$, $A^{sol}\in W^{5,\infty}((0,T)\times\M;T^\ast\M)$, and $\phi\in L^\infty(0,T;W^{7,p}(\M))\cap W^{5,\infty}(0,T;L^\infty(\M))$, we have $\phi\vert_{(0,T)\times\d\M}=0$ and $\delta A^{sol}=0$.
\end{lemma}

\begin{proof}
We fix $\phi$ to be the solution for all $t\in[0,T]$ of the boundary value problem\begin{alignat*}{2}-\Delta_g\phi(t,\cdot)&=-\delta A(t,\cdot)\quad&&\textrm{in }\M,\\\phi(t,\cdot)&=0&&\textrm{on }\d\M.\end{alignat*}
Since $\delta A(t,\cdot)\in W^{5,\infty}(\M)$, according to \cite[Theorem 2.5.1.1]{16}, this problem admits a unique solution $\phi(t,\cdot)\in\bigcap_{p\in[2,\infty)}W^{7,p}(\M)$. Moreover, since $\delta A\in L^\infty(0,T;W^{5,\infty}(\M))$, we also deduce that $\phi\in\bigcap_{p\in[2,\infty)}L^\infty(0,T;W^{7,p}(\M))$. In the same way, using the fact that $\delta A\in W^{5,\infty}(0,T;L^\infty(\M))$, we prove that $\phi\in\bigcap_{p\in[2,\infty)} W^{5,\infty}(0,T;W^{2,p}(\M))$. We then use the Sobolev embedding theorem to deduce that $\phi\in W^{5,\infty}(0,T;L^\infty(\M))$. We fix $A^{sol}=A-d\phi$ and by the Sobolev embedding theorem, deduce that $A^{sol}\in W^{5,\infty}((0,T)\times\M;T^\ast\M)$. Moreover, we see that\[\delta A^{sol}=\delta A-\delta d\phi=\delta A-\Delta_g\phi=0.\]Thus (\ref{3.1}) is the Hodge decomposition of $A$ and the proof of the lemma is complete.
\end{proof}

We start by considering the implication\[\Lambda_{A_1,q_1}=\Lambda_{A_2,q_2}\Rightarrow A^{sol}=0,\]where $A^{sol}$ is the solenoidal part of the Hodge decomposition (\ref{3.1}) of $A$. For this purpose, we establish the following intermediate result.

\begin{lemma}\label{L2}
Let $A_1,A_2\in W^{6,\infty}((0,T)\times\M;T^\ast\M)$ satisfy the matching condition (\ref{1.3}), and fix $A=A_1-A_2$ extended by $0$ on $(0,T)\times(\M_1\setminus\M)$. In particular, for $\tilde{A}_j$ the extension of $A_j$ to $(0,T)\times\M_1$ introduced in the previous section, we have $A=\tilde{A}_1-\tilde{A}_2$. Assuming these conditions are fulfilled, we find that
\begin{equation}\begin{split}\label{3.2}\abs{\int_0^T\int_0^\infty\int_{\d_+S_y\M_1}i(A(r,\theta)\theta)\chi^2(t)h(t,\theta)\exp\Big(i\int_0^\infty A(t,r+s,\theta)\theta ds\Big)d\theta drdt}\\\leq C\big[\norm{\Lambda_{A_1,q_1}-\Lambda_{A_2,q_2}}\lambda^5\tau^{-8}\norm{h}_{H^5((0,T)\times \d_+S_y\M_1)}+\norm{h}_{H^4((0,T)\times \d_+S_y\M_1)}\tau^{-6}\lambda^{-1}\big].\end{split}\end{equation}
\end{lemma}

\begin{proof}
We fix $u_j$, $j=1,2$ the solutions for $j=1,2$ respectively of (\ref{2.2}) taking the form (\ref{2.4}). We write also $\psi_{j,\lambda}=u_j-R_{j,\lambda}$. We consider $v\in H^{1,2}((0,T)\times\M)$ solving
\begin{align*}\begin{aligned}&i\d_tv+\Delta_{g,A_2(t)}v+q_2v=0\ &&\textrm{in }(0,T)\times\M,\\
&v(0,\cdot)=0\ &&\textrm{in }\M,\\
&v=\psi_{1,\lambda}\ &&\textrm{on }(0,T)\times\d\M,\end{aligned}\end{align*}
and consider $w=v-u_1$ which solves
\begin{align*}\begin{aligned}&i\d_tw+\Delta_{g,A_2(t)}w+q_2w=2iA\nabla_gu_1+Vu_1\ &&\textrm{in }(0,T)\times\M,\\
&w(0,\cdot)=0\ &&\textrm{in }\M,\\
&w=0\ &&\textrm{on }(0,T)\times\d\M,\end{aligned}\end{align*}
where $V=i\delta A+\abs{A_2}_g^2-\abs{A_1}_g^2+q_1-q_2$. Multiplying this equation by $\overline{u_2}$ and integrating by parts yields\begin{equation}\label{3.3}\int_0^T\int_\M(2iA\nabla_gu_1+Vu_1)\overline{u_2}dV_g(x)dt=\int_0^T\int_{\d\M}\d_\nu w\overline{u_2}d\sigma_gdt.\end{equation}
Moreover,\[\abs{\int_\Sigma\d_\nu w\overline{u_2}d\sigma_gdt}\leq\norm{(\Lambda_{A_1,q_1}-\Lambda_{A_2,q_2})\psi_{1,\lambda}}_{L^2((0,T)\times\d\M)}\norm{\psi_{2,\lambda}}_{L^2((0,T)\times\d\M)},\] and (\ref{2.16})-(\ref{2.17}) imply
\begin{equation}\begin{split}\label{3.4}\abs{\int_{(0,T)\times\d\M}\d_\nu w\overline{u_2}d\sigma_gdt}\leq C\norm{\Lambda_{A_1,q_1}-\Lambda_{A_2,q_2}}\norm{\psi_{1,\lambda}}_{H^{\frac{9}{4},\frac{3}{2}}((0,T)\times\M)}\norm{\psi_{2,\lambda}}_{L^2((0,T)\times\M)}\\\leq C\norm{\Lambda_{A_1,q_1}-\Lambda_{A_2,q_2}}\lambda^6\norm{h}_{H^5((0,T)\times \d_+S_y\M_1)}\tau^{-8}.\end{split}\end{equation}
Here $C$ is a generic constant which depends only on $\M$, $T$ and $\norm{A_1}_{W^{5,\infty}((0,T)\times\M)}+\norm{A_2}_{W^{5,\infty}((0,T)\times\M)}$. On the other hand, we have that
\begin{equation}\begin{split}\label{3.5}\int_0^T\int_\M(2iA\nabla_gu_1+Vu_1)\overline{u_2}dV_g(x)dt=\\
=\lambda\int_{(0,T)\times\M}2i(A\nabla_g\psi)a_1\overline{a_2}dV_g(x)dt+\lambda\int_{(0,T)\times\M}2i(A\nabla_g\psi)a_1\Big(\frac{\overline{b_2}}{\lambda}+e^{i\lambda(\psi(x)-\lambda t)}\overline{R_{2,\lambda}}\Big)dV_g(x)dt\\+\lambda\int_{(0,T)\times\M}2i(A\nabla_g\psi)\Big(\frac{b_1}{\lambda}+e^{-i\lambda(\psi(x)-\lambda t)}R_{1,\lambda}\Big)\overline{a_2}dV_g(x)dt\\+\lambda\int_{(0,T)\times\M}2i(A\nabla_g\psi)\Big(\frac{b_1}{\lambda}+e^{-i\lambda(\psi(x)-\lambda t)}R_{1,\lambda}\Big)\Big(\frac{\overline{b_2}}{\lambda}+e^{i\lambda(\psi(x)-\lambda t)}\overline{R_{2,\lambda}}\Big)dV_g(x)dt\\+\int_{(0,T)\times\M}\Big(2ie^{i\lambda(\psi(x)-\lambda t)}A\Big(\nabla_ga_1+\frac{\nabla_gb_1}{\lambda}+\nabla_gR_{1,\lambda}\Big)+Vu_1\Big)\overline{u_2}dV_g(x)dt.\end{split}\end{equation}
We then divide (\ref{3.5}) by $\lambda$ and apply (\ref{2.20})-(\ref{2.21}) to obtain
\[\begin{split}\abs{\int_{(0,T)\times\M}i(A\nabla_g\psi)a_1\overline{a_2}dV_g(x)dt}\\\leq\lambda^{-1}\abs{\int_{(0,T)\times\M}(2iA\nabla_gu_1+Vu_1)\overline{u_2}dV_g(x)dt}+C\norm{h}_{H^4((0,T)\times \d_+S_y\M_1)}\tau^{-6}\lambda^{-1}.\end{split}\]
Using polar normal coordinates in the left hand side of the above gives us\[\begin{split}\abs{\int_0^T\int_0^\infty\int_{\d_+S_y\M_1}i(A(t,r,\theta)\theta)\chi^2(t)h(t,\theta)\mu(r,\theta)^{-\frac{1}{2}}\exp\Big(i\int_0^\infty A(t,r+s,\theta)\theta ds\Big)dV_{g}(r,\theta)dt}\\\leq\lambda^{-1}\abs{\int_{(0,T)\times\M}(2iA\nabla_gu_1+Vu_1)\overline{u_2}dV_g(x)dt}+C\norm{h}_{H^4((0,T)\times \d_+S_y\M_1}\tau^{-6}\lambda^{-1}.\end{split}\]
Using now the fact that $\mu(r,\theta)^{-\frac{1}{2}}dV_{g}(r,\theta)=drd\theta$, we conclude that
\[\begin{split}\abs{\int_0^T\int_0^\infty\int_{\d_+S_y\M_1}i(A(t,r,\theta)\theta)\chi^2(t)h(t,\theta)\exp\Big(i\int_0^\infty A(t,r+s,\theta)\theta ds\Big)d\theta drdt}\\\leq\lambda^{-1}\abs{\int_{(0,T)\times\M}(2iA\nabla_gu_1+Vu_1)\overline{u_2}}dV_g(x)dt+C\norm{h}_{H^4((0,T)\times \d_+S_y\M_1)}\tau^{-6}\lambda^{-1}.\end{split}\]
We use this last estimate together with (\ref{3.3}) and (\ref{3.4}) to obtain (\ref{3.2}).\end{proof}

Armed with the above, we are now in a position to complete the proof of the uniqueness result.
\begin{proof}[Proof of Theorem \ref{T1}]
Let us assume that $\Lambda_{A_1,q_1}=\Lambda_{A_2,q_2}$, and begin by proving that this condition implies that $A^{sol}=0$. We recall also Definition \ref{d2} of $I_1$, the geodesic ray transform for $1$-forms given by (\ref{I1}). According to s-injectivity of the transform $I_1$ (consult e.g. \cite{1} or \cite[Theorem 4]{32}), it is enough to show that $I_1A(t,\cdot)=0$. Then, sending $\lambda\rightarrow\infty$ in (\ref{3.2}) we obtain\begin{equation}\label{3.6}\int_0^T\int_0^\infty\int_{\d_+S_y\M_1}i(A(t,r,\theta)\theta)\chi^2(t)h(t,\theta)\exp\Big(i\int_0^\infty A(t,r+s,\theta)\theta ds\Big)d\theta drdt=0.\end{equation}
On the other hand, notice that, due to (\ref{2c1}), for $A=\sum_{j=1}^na_jdx^j$ we have
\[\begin{split}\int_0^\infty A(t,r,\theta)\theta dr=&\int_0^{\tau_+(y,\theta)}A(t,r,\theta)\theta dr\\=&\int_0^{\tau_+(y,\theta)}A(t,\gamma_{y,\theta}(s))\gamma'_{y,\theta}(s)ds=I_1[A(t,\cdot)](y,\theta).\end{split}\]
Thus we deduce that\[\begin{split}\int_0^\infty i(A(t,r,\theta)\theta)\exp\Big(i\int_0^\infty A(t,r+s,\theta)\theta ds\Big)dr=&\int_0^\infty i(A(t,r,\theta)\theta)\exp\Big(i\int_r^\infty A(t,s,\theta)\theta ds\Big)dr\\=&-\int_0^\infty\d_r\exp\Big(i\int_r^\infty A(t,s,\theta)\theta ds\Big)dr\\=&\exp\Big(i\int_0^\infty A(t,s,\theta)\theta ds\Big)-1=e^{iI_1[A(t,\cdot)](y,\theta)}-1.\end{split}\]
Using this identity in (\ref{3.6}) and applying Fubini's theorem, we get\[\int_0^T\int_{\d_+S_y\M_1}\chi^2(t)\big[e^{iI_1[A(t,\cdot)](y,\theta)}-1\big]h(t,\theta)d\theta dt=0.\]Since $h\in C_0^\infty((0,T)\times\d_+S\M_1)$ is arbitrary, we deduce that\[\chi^2(t)\big[e^{iI_1[A(t,\cdot)](y,\theta)}-1\big]=0,\quad t\in(0,T),\ (y,\theta)\in\d_+S\M_1.\]But since $\tau\in(0,\frac{T}{4})$ is arbitrary and $\chi(t)=1$ for $t\in[2\tau,T-2\tau]$, we see that\[e^{iI_1[A(t,\cdot)](y,\theta)}=1,\quad t\in[0,T],\ (y,\theta)\in\d_+S\M_1,\] and hence deduce that for all $t\in[0,T]$, $I_1[A(t,\cdot)](y,\theta)\in2\pi\Z$. Since $A\in W^{6,\infty}((0,T)\times\M_1;T^\ast\M_1)$ one can check that $I_1A\in C([0,T]\times\d_+S\M_1)$. Then since for all $y\in\d\M_1$ it holds that $\d_+S_y\M_1$ is connected, we conclude that the map $[0,T]\times \d_+S_y\M_1\ni(t,\theta)\mapsto I_1[A(t,\cdot)](y,\theta)$ is constant. On the other hand, note that $A=0$ on $\M_1\setminus\M$, so that for any $y\in\d\M_1$ there exists $\theta\in \d_+S_y\M_1$ such that for all $t\in[0,T]$ we have $I_1[A(t,\cdot)](y,\theta)=0$. Therefore we conclude that $A^{sol}=0$.\\

We can then use the Hodge decomposition (\ref{3.1}), to deduce the existence of $\phi\in W^{5,\infty}((0,T)\times\M)$ satisfying $\phi\vert_{(0,T)\times\d\M}=0$ such that $A_2=A_1+d\phi$. Thus the proof will be completed if we show that $q_2=q_1-\d_t\phi$. Since $A_2=A_1+d\phi$ we can put $q_3=q_1-\d_t\phi$ and by gauge invariance we have $\Lambda_{A_1,q_1}=\Lambda_{A_2,q_3}$. Thus, by assumption it follows that\begin{equation}\label{3.7}\Lambda_{A_2,q_3}=\Lambda_{A_1,q_1}=\Lambda_{A_2,q_2}.\end{equation}

Therefore, the proof will be complete if we prove that condition (\ref{3.7}) implies that $q_3=q_2$. For this purpose, we let $y\in\d\M_1$, $h\in C_0^\infty((0,T)\times \d_+S_y\M_1)$. We consider $u_2$ the solution of (\ref{2.2}) for $j=2$ taking the form (\ref{2.4}), and $u_1$ the solution of (\ref{2.2}) but with $A_j$ replaced by $A_2$ and $q_j$ replaced by $q_3$, again taking the form (\ref{2.2}). Note that $q_3=q_1-\d_t\phi\in W^{4,\infty}((0,T)\times\M)$, so this construction is still valid. In particular, taking $A_1=A_2$ in (\ref{3.3}) we obtain\[\int_0^T\int_\M(q_3-q_2)u_1\overline{u_2}dV_g(x)dt=\int_0^T\int_{\d\M}\big[(\Lambda_{A_2,q_3}-\Lambda_{A_2,q_2})\psi_{1,\lambda}\big]\overline{u_2}d\sigma_gdt=0.\]

Fixing $q=q_3-q_2$ extended by $0$ on $(0,T)\times(\M_1\setminus\M)$, we get \[\begin{split}\int_0^T\int_\M qu_1\overline{u_2}dV_g(x)dt=&\int_{(0,T)\times\M}qa_1\overline{a_2}dV_g(x)dt+\int_{(0,T)\times\M}qa_1\Big(\frac{\overline{b_2}}{\lambda}+e^{i\lambda(\psi(x)-\lambda t)}\overline{R_{2,\lambda}}\Big)dV_g(x)dt\\+&\int_{(0,T)\times\M}q\Big(\frac{b_1}{\lambda}+e^{-i\lambda(\psi(x)-\lambda t)}R_{1,\lambda}\Big)\overline{a_2}dV_g(x)dt\\+&\int_{(0,T)\times\M}q\Big(\frac{b_1}{\lambda}+e^{-i\lambda(\psi(x)-\lambda t)}R_{1,\lambda}\Big)\Big(\frac{\overline{b_2}}{\lambda}+e^{i\lambda(\psi(x)-\lambda t)}\overline{R_{2,\lambda}}\Big)dV_g(x)dt.\end{split}\]
Then, we argue similarly to the proof of Lemma \ref{L2}. Using polar normal coordinates and (\ref{2.20})-(\ref{2.21}) we get\[\abs{\int_0^T\int_0^\infty\int_{\d_+S_y\M_1}\chi^2(t)q(t,r,\theta)\overline{h(t,\theta)}d\theta drdt}\leq C\norm{h}_{H^4((0,T)\times \d_+S_y\M_1)}\tau^{-6}\lambda^{-1}.\]

And we send $\lambda\rightarrow\infty$ to obtain\begin{equation}\label{3.8}\int_0^T\int_0^\infty\int_{\d_+S_y\M_1}\chi^2(t)q(t,r,\theta)\overline{h(t,\theta)}d\theta drdt=0.\end{equation}

Let us recall the definition of the geodesic ray transform $I_0$ acting on functions, given by (\ref{I0}). In light of (\ref{3.8}), we allow $y\in\d\M$ and $h\in C_0^\infty((0,T)\times \d_+S_y\M_1)$ to be arbitrary, whence we deduce that\[\chi^2(t)I_0[q(t,\cdot)](y,\theta)=\int_0^{\tau_+(y,\theta)}\chi^2(t)q(t,r,\theta)dr=0,\quad t\in(0,T),\ (y,\theta)\in\d_+S\M_1.\]

Now, since $\tau\in(0,\frac{T}{4})$ is arbitrary and $\chi=1$ on $[2\tau,T-2\tau]$, we conclude that $I_0[q(t,\cdot)]=0$ for all $t\in(0,T)$. Then by injectivity of $I_0$ on $L^2(\M)$ (e.g. \cite[Theorem 3]{32}) implies that $q=0$, whence $q_2=q_3=q_1-\d_t\phi$. This completes the proof of Theorem \ref{T1}.\end{proof}

\section{Stable Determination of the Magnetic Potential}\label{section5}
In this section we establish the stability estimate in the recovery of the magnetic potential stated in Theorem \ref{T2}. For $j=1,2$, we assume that $A_j\in W^{6,\infty}((0,T)\times\M;T^\ast\M)\cap H^{3n+4}((0,T)\times\M;T^\ast\M)$ fulfill \eqref{1.333}. Then, for $A=A_1-A_2$ extended by $0$ on $(0,T)\times (\M_1\setminus\M)$ we have $A\in W^{6,\infty}((0,T)\times\M_1;T^\ast\M_1)\cap H^{3n+4}((0,T)\times\M_1;T^\ast\M_1)$. We will also assume for the moment that for some small $\varepsilon>0$ it holds that\begin{equation}\label{Small}\norm{A^{sol}}_{L^2((0,T)\times\M_1)}\leq\varepsilon.\end{equation}

Before proving Theorem \ref{T2}, let us recall some facts about the geodesic ray transform $I_1$.\\
Firstly, according to \cite[Theorem 4.2.1]{28}, the ray transform for $1$-forms extends to a bounded linear operator $I_1:H^k(\M_1;T^\ast\M_1)\rightarrow H^k(\d_+S\M_1)$. Fixing $w(x,\theta)=\abs{\pair{\theta,\nu(x)}_g}$, we can also extend $I_1$ to a bounded linear operator $I_1:L^2(\M_1;T^\ast\M_1)\rightarrow L^2_w(\d_+S\M_1)$, where $L^2_w(\d_+S\M_1)$ is the $L^2$ space with respect to the weighted measure $w(y,\theta)d\theta d\sigma_g(y)$, and thus define $I_1^\ast:L^2_w(\d_+S\M_1)\rightarrow L^2(\M_1;T^\ast\M_1)$ as the adjoint of $I_1$. By condition (\ref{1.3}) we have $A\in H^5((0,T)\times\M_1;T^\ast\M_1)$ with $\textrm{supp}\ A(t,\cdot)\subset\M$ for $t\in(0,T)$. Moreover, according to \cite[Section 8]{32}, the operator $I_1^\ast I_1$,  is an elliptic pseudodifferential operator of order $-1$. Together with condition (\ref{C1}), we have for $0\leq k\leq5$\begin{equation}\label{5.4}\norm{I_1^\ast I_1A}_{H^k((0,T)\times\M_1;T^\ast\M_1)}\leq C\norm{A}_{H^k((0,T)\times\M_1;T^\ast\M_1)}\leq CB.\end{equation}
Also according to \cite[Section 8]{32}, we can find constants $C_1,C_2>0$ such that for $0\leq k\leq5$\begin{equation}\label{5e1}C_1\norm{A^{sol}}_{L^2(0,T;H^k(\M_1))}\leq \norm{I_1^\ast I_1A}_{L^2(0,T;H^{k+1}(\M_1))}\leq C_2\norm{A^{sol}}_{L^2(0,T;H^k(\M_1))}.\end{equation}.
\begin{proof}[Proof of Theorem \ref{T2} subject to (\ref{Small})]
Following the work of the previous section, we allow $h(t,\theta)$ to depend on $y\in\d\M_1$. We can rewrite inequality (\ref{3.2}) in the form
\begin{equation}\begin{split}\label{2}\abs{\int_0^T\int_{\d_+S_y\M_1}\big(e^{iI_1[A(t,\cdot)](y,\theta)}-1\big)\chi^2(t)h(t,y,\theta)d\theta dt}\\\leq C\Big(\norm{\Lambda_{A_1,q_1}-\Lambda_{A_2,q_2}}\lambda^5\tau^{-8}\norm{h(y,\cdot)}_{H^5((0,T)\times \d_+S_y{\M_1})}+\lambda^{-1}\tau^{-6}\norm{h(y,\cdot)}_{H^4((0,T)\times \d_+S_y\M_1)}\Big).\end{split}\end{equation}
We can use the Taylor expansion $e^t=1+t+t^2\int_0^1e^{st}(1-s)ds$ to see that \[e^{iI_1[A(t,\cdot)](y,\theta)}-1=iI_1[A(t,\cdot)](y,\theta)-I_1[A(t,\cdot)]^2(y,\theta)\int_0^1e^{isI_1[A(t,\cdot)](y,\theta)}(1-s)ds,\]and using this identity in (\ref{2}) yields
\[\begin{aligned}&\abs{\int_0^T\chi^2(t)\int_{\d_+S_y\M_1}I_1[A(t,\cdot)](y,\theta)h(t,y,\theta)d\theta dt}\leq C\Big(\norm{\Lambda_{A_1,q_1}-\Lambda_{A_2,q_2}}\lambda^5\tau^{-8}\norm{h(y,\cdot)}_{H^5((0,T)\times \d_+S_y{\M_1})}\\&+\lambda^{-1}\tau^{-6}\norm{h(y,\cdot)}_{H^4((0,T)\times \d_+S_y\M_1)}+\norm{h(y,\cdot)}_{L^2((0,T)\times \d_+S_y\M_1)}\norm{I_1A}^2_{\mathcal{C}^0([0,T]\times\d_+S\M)}\Big).\end{aligned}\]
Combining this with the fact that
\[I_1A=I_1d\phi+I_1A_{sol}=I_1A_{sol}\]
and the definition of $I_1$, we deduce that 
\[\norm{I_1A}_{\mathcal{C}^0([0,T]\times\d_+S\M)}\leq C\norm{A^{sol}}_{\mathcal{C}^0([0,T]\times\M;T^\ast\M)}.\]
This implies that
\begin{equation}\begin{split}\label{3}\abs{\int_0^T\chi^2(t)\int_{\d_+S_y\M_1}I_1[A(t,\cdot)](y,\theta)h(t,y,\theta)d\theta dt}\leq C\Big(\norm{\Lambda_{A_1,q_1}-\Lambda_{A_2,q_2}}\lambda^5\tau^{-8}\norm{h(y,\cdot)}_{H^5((0,T)\times \d_+S_y{\M_1})}\\+\lambda^{-1}\tau^{-6}\norm{h(y,\cdot)}_{H^4((0,T)\times \d_+S_y\M_1)}+\norm{h(y,\cdot)}_{L^2((0,T)\times \d_+S_y\M_1)}\norm{A^{sol}}^2_{\mathcal{C}^0([0,T]\times\M;T^\ast\M)}\Big).\end{split}\end{equation}
Since $I_1$ extends to a bounded linear operator $I_1:H^k(\M_1;T^\ast\M_1)\rightarrow H^k(\d_+S\M_1)$, we can choose $h(t,y,\theta)=I_1I_1^\ast I_1[A(t,\cdot)](y,\theta)$ and then integrate (\ref{3}) with respect to the volume form $d\sigma_g$ of $\d\M_1$. Using the compactness of $\M_1$ we deduce that\begin{equation}\begin{split}\label{a}\int_0^T\chi^2(t)\int_{\M_1}\abs{I_1^\ast I_1[A(t,\cdot)](x)}^2dV_g(x)dt=\int_0^T\chi^2(t)\int_{\d_+S\M_1}I_1[A(t,\cdot)](y,\theta)h(t,\theta)\abs{\pair{\theta,\nu(y)}_g}d\theta d\sigma_g(y)dt\\\leq C\Big(\norm{\Lambda_{A_1,q_1}-\Lambda_{A_2,q_2}}\lambda^5\tau^{-8}\norm{I_1^\ast I_1A}_{H^5((0,T)\times\M_1;T^\ast\M_1)}+\lambda^{-1}\tau^{-6}\norm{I_1^\ast I_1A}_{H^4((0,T)\times\M_1;T^\ast\M_1)}\\+\norm{I_1^\ast I_1A^{sol}}_{L^2((0,T)\times\M_1;T^\ast\M_1)}\norm{A^{sol}}^2_{\mathcal{C}^0([0,T]\times\M;T^\ast\M)}\Big).\end{split}\end{equation}
Moreover, using (\ref{5.4}) we can further simplify (\ref{a}) in order to obtain\begin{equation}\begin{split}\label{4}\int_0^T\chi^2(t)\int_{\M_1}\abs{I_1^\ast I_1[A(t,\cdot)](x)}^2dV_g(x)dt\\\leq C\Big(\norm{\Lambda_{A_1,q_1}-\Lambda_{A_2,q_2}}\lambda^5\tau^{-8}+\lambda^{-1}\tau^{-6}+\norm{A^{sol}}_{L^2((0,T)\times\M_1;T^\ast\M_1)}\norm{A^{sol}}^2_{\mathcal{C}^0([0,T]\times\M;T^\ast\M)}\Big).\end{split}\end{equation}
Since we also have\begin{equation}\begin{split}\label{b}\abs{\int_0^T\chi^2(t)\int_{\M_1}\abs{I_1^\ast I_1[A(t,\cdot)](x)}^2dV_g(x)dt-\int_0^T\int_{\M_1}\abs{I_1^\ast I_1[A(t,\cdot)](x)}^2dV_g(x)dt}\\\leq C\Big[\int_0^\tau(1-\chi^2(t))dt+\int_{T-\tau}^T(1-\chi^2(t))dt\Big]\leq C\tau,\end{split}\end{equation} we obtain the estimate\begin{equation}\begin{split}\label{edit1}\int_0^T\int_{\M_1}\abs{I_1^\ast I_1[A(t,\cdot)](x)}^2dV_g(x)dt\\\leq C\Big(\norm{\Lambda_{A_1,q_1}-\Lambda_{A_2,q_2}}\lambda^5\tau^{-8}+\lambda^{-1}\tau^{-6}+\tau+\norm{A^{sol}}_{L^2((0,T)\times\M_1)}\norm{A^{sol}}^2_{\mathcal{C}^0([0,T]\times\M_1)}\Big).\end{split}\end{equation}
We now set $\gamma_\ast=\min\big((\frac{T}{4})^{44},1\big)$. Let $\gamma=\norm{\Lambda_{A_1,q_1}-\Lambda_{A_2,q_2}}$. For $\gamma<\gamma_\ast$, we can choose $\tau=\gamma^\frac{1}{44},\ \lambda=\tau^{-7}$, and deduce that\begin{equation}\label{5}\norm{I_1^\ast I_1A}_{L^2((0,T)\times\M_1)}^2\leq C\Big(\gamma^{\frac{1}{44}}+\norm{A^{sol}}_{L^2((0,T)\times\M_1)}\norm{A^{sol}}^2_{\mathcal{C}^0([0,T]\times\M_1)}\Big).\end{equation}
By the Sobolev embedding theorem, interpolation, and condition (\ref{C1}), we observe that\begin{equation}\begin{split}\label{6}\norm{A^{sol}}_{\mathcal{C}^0([0,T]\times\M_1)}\leq C\norm{A^{sol}}_{H^{\frac{n+1}{2}+\frac{1}{6}}((0,T)\times\M_1)}\\\leq C\norm{A^{sol}}^{\frac{5}{6}}_{L^2((0,T)\times\M_1)}\norm{A^{sol}}^{\frac{1}{6}}_{H^{3n+4}((0,T)\times\M_1)}\leq C\norm{A^{sol}}^{\frac{5}{6}}_{L^2((0,T)\times\M_1)}.\end{split}\end{equation}
Then, using (\ref{5e1}) and condition (\ref{C1}), interpolation also yields the estimate\begin{equation}\begin{split}\label{7}\norm{A^{sol}}_{L^2((0,T)\times\M_1)}^2&\leq C\norm{I_1^\ast I_1A}^2_{L^2(0,T;H^1(\M_1))}\leq C\norm{I_1^\ast I_1A}^{\frac{10}{6}}_{L^2((0,T)\times\M_1)}\norm{I_1^\ast I_1A}^{\frac{1}{3}}_{L^2(0,T;H^6(\M_1))}\\&\leq C\norm{I_1^\ast I_1A}^{\frac{10}{6}}_{L^2((0,T)\times\M_1)}\norm{A^{sol}}^{\frac{1}{3}}_{L^2(0,T;H^5(\M_1))}\leq C\norm{I_1^\ast I_1A}^{\frac{10}{6}}_{L^2((0,T)\times\M_1)}.\end{split}\end{equation}
Finally we combine (\ref{5}), (\ref{6}) and (\ref{7}) to obtain\[\begin{split}\norm{A^{sol}}_{L^2((0,T)\times\M_1)}^2\leq C\norm{I_1^\ast I_1A}_{L^2}^{\frac{10}{6}}\leq C\gamma^{\frac{5}{264}}+C\norm{A^{sol}}^{\frac{80}{36}}_{L^2}\leq C\gamma^{\frac{5}{264}}+C\varepsilon^{\frac{8}{36}}\norm{A^{sol}}^2_{L^2((0,T)\times\M_1)}.\end{split}\]
Thus for small $\varepsilon$ we deduce that\[\norm{A^{sol}}_{L^2((0,T)\times\M)}\leq C\gamma^{\frac{5}{528}}.\]
Similarly for $\gamma\geq\gamma_\ast$, we have\begin{equation}\label{5e2}\norm{A^{sol}}_{L^2((0,T)\times\M)}\leq \frac{\norm{A^{sol}}_{L^2((0,T)\times\M)}\gamma^{\frac{5}{528}}}{\gamma_\ast^{\frac{5}{528}}}\leq C\gamma^{\frac{5}{528}}.\end{equation}
Thus the proof of Theorem \ref{T2} is complete, subject to the smallness assumption (\ref{Small}).\end{proof}
We will now show that the assumption that (\ref{Small}) holds a priori is unnecessary. Define $\eta\in\mathcal{C}^\infty(\R^{n})$ by \[\eta(x)=\begin{cases}C\exp(\frac{1}{\abs{x}^2-1})&\textrm{ if }\abs{x}<1,\\0&\textrm{ if }\abs{x}\geq1,\end{cases}\]where $C>0$ is chosen so that $\int_{\R^{n}}\eta(x)dx=1$. We further define the function\[\eta_\rho(x)=\frac{1}{\rho^{n}}\eta\Big(\frac{x}{\rho}\Big).\]
Note that $\eta_\rho$ approximates the Dirac delta distribution on $\R^n$ as $\rho\rightarrow0$. Arguing as we did in (\ref{b}), we use the estimate (\ref{2}) to deduce that\begin{equation}\begin{split}\label{8}\abs{\int_0^T\int_{\d_+S_y\M_1}\big(e^{iI_1[A(t,\cdot)](y,\theta)}-1\big)h(t,y,\theta)d\theta dt}\\\leq C\Big(\norm{\Lambda_{A_1,q_1}-\Lambda_{A_2,q_2}}\lambda^5\tau^{-8}\norm{h(y,\cdot)}_{H^5((0,T)\times \d_+S_y{\M_1})}+\lambda^{-1}\tau^{-6}\norm{h(y,\cdot)}_{H^4((0,T)\times \d_+S_y\M_1)}+\tau\Big).\end{split}\end{equation}
Since $A$ is extended by $0$ to $(0,T)\times(\M_1\setminus\M)$, it follows that $e^{iI_1[A(t,\cdot)](y,\theta)}-1$ is compactly supported in $[0,T]\times \d_+S_y\M_1$. We can find a finite open cover $\{U_i\}_{i=1}^N$ of $\d\M_1$ so that for all $y\in U_i$ we can choose the same spherical coordinates $\theta:=\R^{n-1}\ni\alpha\mapsto\theta(\alpha)$ on $S_y\M_1$ in such a way that $\theta(\alpha)$ gives coordinates in a neighborhood of $\textrm{supp}(e^{iI_1[A(t,\cdot)](y,\theta)}-1)\subset \d_+S_y\M_1$.\\

We can then fix $y\in\d\M_1$, $\theta_0\in \d_+S_y\M_1$, $t_0\in(0,T)$. Let $\alpha_0=\alpha(\theta_0)$, $\gamma=\norm{\Lambda_{A_1,q_1}-\Lambda_{A_2,q_2}}$. We define the function $f(\alpha,t)=e^{iI_1[A(t,\cdot)](y,\theta(\alpha))}-1$ and let $h(t,y,\theta)$ approximate the cylindrical Dirac delta distribution, that is\[h(t,y,\theta(\alpha))=\frac{1}{sin^{n-2}(\alpha_1)sin^{n-3}(\alpha_2)\cdots sin(\alpha_{n-2})}\eta_\rho\big((\alpha_0,t_0)-(\alpha,t)\big).\]
It is well known (see for instance \cite[Lemma 2.1]{Sa1}) that
\[\norm{h}_{H^k((0,T)\times \d_+S_y\M_1)}\leq \rho^{-(n+k)},\quad k\in\N.\]
In addition, we fix
\[f^\rho(\alpha_0,t_0)=\int_{\R^n}f(\alpha,t)h\big[\big(t_0,y,\theta(\alpha_0)\big)-\big(t,y,\theta(\alpha)\big)\big]dtd\alpha.\]
We use (\ref{8}) to deduce that\begin{equation}\label{9}\abs{\int_{\R^{n}}f(\alpha,t)\eta_\rho\big((\alpha_0,t_0)-(\alpha,t)\big)dtd\alpha}\leq C\Big(\gamma\lambda^5\tau^{-8}\rho^{-n-5}+\lambda^{-1}\tau^{-6}\rho^{-n-4}+\tau\Big).\end{equation}
In particular, $C$ is a positive constant depending only on $\M$, $T$ and $B$, and independent of $y$. In order to deal with the left hand side above, we need the following Lemma:

\begin{lemma}\label{l2}
Let $f:\R^{n}\mapsto\R$ be $\mathcal{C}^1$, and let $f^\rho(x_0)=\int_{B(x_0,\rho)}f(x)\eta_\rho(x_0-x)dx$. Then for any $x_0\in\R^{n}$ we have that\[\abs{f^\rho(x_0)-f(x_0)}\leq C\norm{f}_{\mathcal{C}^1}\rho.\]
\end{lemma}

\begin{proof}
\[\begin{split}\abs{f^\rho(x_0)-f(x_0)}=\abs{\int_{B(x_0,\rho)}\eta_\rho(x_0-x)[f(x)-f(x_0)]dx}\leq\int_{B(x_0,\rho)}\eta_\rho(x_0-x)\abs{f(x)-f(x_0)}dx\\\leq\int_{B(x_0,\rho)}\eta_\rho(x_0-x)\norm{f}_{\mathcal{C}^1}\rho dx\leq C\norm{f}_{\mathcal{C}^1}\cdot\rho.\end{split}\]
\end{proof}
Since $I_1:\mathcal{C}^k(\M_1;T^\ast\M_1)\mapsto\mathcal{C}^k(\d S\M_1)$ is bounded, $\norm{A}_{W^{5,\infty}((0,T)\times\M_1;T^\ast\M_1)}\leq B$, then we must have $\norm{f}_{\mathcal{C}^1}\leq CB$ when $f(\alpha,t)=e^{iI_1[A(t,\cdot)](y,\theta(\alpha))}-1$. Thus, Lemma (\ref{l2}) together with (\ref{9}) tells us that\[\abs{e^{iI_1[A(t_0,\cdot)](y,\theta_0)}-1}\leq C\Big(\gamma\lambda^5\tau^{-8}\rho^{-n-5}+\lambda^{-1}\tau^{-6}\rho^{-n-4}+\tau+\rho\Big).\]
For $\gamma\leq\min\big((\frac{T}{4})^{6n+69},1\big)$ we can choose $\tau=\gamma^{\frac{1}{6n+69}}$, $\lambda=\tau^{-n-11}$, $\rho=\tau$ to deduce that\[\abs{e^{iI_1[A(t_0,\cdot)](y,\theta_0)}-1}\leq C\gamma^{\frac{1}{6n+69}},\]with $C$ independent of $y$. We now choose $\gamma_0$ small enough so the right hand side is near $0$ when $\gamma<\gamma_0$. But this implies that $I_1[A(t_0,\cdot)](y,\theta_0)$ remains close to integer multiples of $2\pi$ whenever $\gamma<\gamma_0$. Recall that $A$ is extended to $(0,T)\times\M_1\setminus\M$ by zero. Thus, for choices of $y,\theta_0$ corresponding to short geodesics remaining close to the boundary of $\M_1$, we have $I_1[A(t_0,\cdot)](y,\theta_0)=0$. Then, the continuity of $I_1[A(t_0,\cdot)]$ in $y,\theta_0$, together with the previous argument implies $I_1[A(t_0,\cdot)](y,\theta_0)$ is close to zero when $\gamma<\gamma_0$. But $\norm{I_1A}_{\mathcal{C}^0([0,T]\times\d_+S\M_1)}\leq\varepsilon^2$ implies $\norm{I_1A}_{L^2(0,T)\times\d_+S\M_1}\leq C\varepsilon^2$, and in turn $\norm{I_1^\ast I_1A}_{L^2((0,T)\times\M_1;T^\ast\M_1)}\leq C\varepsilon^2$.\\
Then interpolation gives\[\begin{split}\norm{A^{sol}}_{L^2((0,T)\times\M_1)}\leq C\norm{I_1^\ast I_1A}_{L^2(0,T;H^1(\M_1))}\leq C\norm{I_1^\ast I_1A}_{L^2((0,T)\times\M_1)}^{\frac{1}{2}}\norm{I_1^\ast I_1A}_{L^2(0,T;H^2(\M_1))}^{\frac{1}{2}}\\\leq C\norm{I_1^\ast I_1A}_{L^2((0,T)\times\M_1)}^{\frac{1}{2}}\norm{A^{sol}}_{L^2(0,T;H^1(\M_1))}^{\frac{1}{2}}\leq C\norm{I_1^\ast I_1A}_{L^2((0,T)\times\M_1)}^{\frac{1}{2}}\leq C\varepsilon.\end{split}\]
Thus, for $\gamma<\gamma_0$ we conclude that the smallness assumption $\norm{A^{sol}}_{L^2((0,T)\times\M_1)}\leq\varepsilon$ holds. Therefore, we can rerun the argument of the previous section with $\gamma_\ast$ replaced by $\gamma_0$, and reach the same conclusion without the need to assume smallness a priori. On the other hand, if $\gamma\geq\gamma_0$, we proceed as in (\ref{5e2}). With this, the proof of Theorem \ref{T2} is now complete.

\section{Stable Recovery of the Electric Potential}\label{section6}
This section is devoted to proving the stability estimate in the recovery of the electric potential stated in Theorem \ref{T3}. Henceforth, for $j=1,2$ we assume that $A_j\in W^{5,\infty}((0,T)\times\M_1;T^\ast\M_1)$ with $\delta A_1=\delta A_2$ (so that $A=A^{sol}$), $q_j\in W^{4,\infty}((0,T)\times\M_1)$ and that conditions (\ref{1.4}) and (\ref{1.5}) are fulfilled. Additionally, we continue to assume that condition (\ref{C1}) holds true for the magnetic potential. In light of (\ref{2.16})-(\ref{2.21}), we can use (\ref{3.3})-(\ref{3.4}) to deduce that\begin{equation}\begin{split}\label{6.1}\abs{\int_0^T\int_\M Vu_1\overline{u_2}dV_g(x)dt}\leq C\Big(\lambda\tau^{-6}\norm{A}_{L^\infty((0,T)\times\M;T^\ast\M)}\norm{h}_{H^4((0,T)\times\d_+S_y\M)}\\+\gamma\tau^{-8}\lambda^{6}\norm{h}_{H^5((0,T)\times\d_+S_y\M_1)}\Big),\end{split}\end{equation}where again $\gamma$ denotes $\norm{\Lambda_{A_1,q_1}-\Lambda_{A_2,q_2}}$. Using the fact that\[\begin{split}\int_0^T\int_\M Vu_1\overline{u_2}\ dV_g(x)dt=\int_0^T\int_\M Va_1\overline{a_2}\ dV_g(x)dt+\int_0^T\int_\M Va_1\Big(\frac{\overline{b_2}}{\lambda}+e^{i\lambda(\psi-\lambda t)}\overline{R_{2,\lambda}}\Big)\ dV_g(x)dt\\+\int_0^T\int_\M V\Big(\frac{b_1}{\lambda}+e^{-i\lambda(\psi-\lambda t)}R_{1,\lambda}\Big)\overline{a_2}\ dV_g(x)dt\\+\int_0^T\int_\M V\Big(\frac{b_1}{\lambda}+e^{-i\lambda(\psi-\lambda t)}R_{1,\lambda}\Big)\Big(\frac{\overline{b_2}}{\lambda}+e^{i\lambda(\psi-\lambda t)}\overline{R_{2,\lambda}}\Big)\ dV_g(x)dt\end{split},\]together with (\ref{6.1}) and (\ref{2.16})-(\ref{2.21}), we obtain\begin{equation}\begin{split}\label{6.2}\int_0^T\int_\M Va_1\overline{a_2}\ dV_g(x)dt\leq& C\Big(\lambda\tau^{-6}\norm{A}_{L^\infty((0,T)\times\M;T^\ast\M)}\norm{h}_{H^4((0,T)\times\d_+S_y\M)}\\&+\gamma\tau^{-8}\lambda^{6}\norm{h}_{H^5((0,T)\times\d_+S_y\M_1)}+\lambda^{-1}\tau^{-4}\norm{h}_{H^4((0,T)\times\d_+S_y\M_1)}\Big).\end{split}\end{equation}
Then, by the definition of $V$ together with Stokes' theorem, we deduce\[\int_0^T\int_\M Va_1\overline{a_2}\ dV_g(x)dt=\int_0^T\int_\M qa_1\overline{a_2}\ dV_g(x)dt-i\int_0^T\int_\M A\nabla_g(a_1\overline{a_2})\ dV_g(x)dt-\int_0^T\int_M\pair{A,A_1+A_2}_ga_1\overline{a_2}\ dV_g(x)dt,\]whence we have\begin{equation}\begin{split}\label{6.3}\int_0^T\int_\M qa_1\overline{a_2}\ dV_g(x)dt\leq& C\Big(\lambda\tau^{-6}\norm{A}_{L^\infty((0,T)\times\M;T^\ast\M)}\norm{h}_{H^4((0,T)\times\d_+S_y\M)}\\&+\gamma\tau^{-8}\lambda^{6}\norm{h}_{H^5((0,T)\times\d_+S_y\M_1)}+\lambda^{-1}\tau^{-4}\norm{h}_{H^4((0,T)\times\d_+S_y\M_1)}\Big).\end{split}\end{equation}
Since it holds that\[\int_0^T\int_\M qa_1\overline{a_2}\ dV_g(x)dt=\int_0^T\int_{\d_+S_y\M_1}\int_0^\infty q(t,r,\theta)\chi^2(t)h(t,\theta)\exp\Big(i\int_0^\infty A(t,r+s,\theta)\theta ds\Big)drd\theta dt\]we deduce\[\begin{split}\abs{\int_0^T\int_{\d_+S_y\M_1}\int_0^\infty\chi^2(t)q(t,r,\theta)h(t,\theta)drd\theta dt}\leq\abs{\int_0^T\int_\M qa_1\overline{a_2}dV_g(x)dt}\\+\abs{\int_0^T\int_{\d_+S_y\M_1}\int_0^\infty\chi^2(t)q(t,r,\theta)h(t,\theta)\Big[\exp\Big(i\int_0^\infty A(t,r+s,\theta)\theta ds\Big)-1\Big]drd\theta dt}.\end{split}\]
Applying the mean value theorem to the second term on the right, we find that\[\abs{\int_0^T\int_\M I_0[q(t,\cdot)](y,\theta)\chi^2(t)h(t,y,\theta)d\theta dt}\leq\abs{\int_0^T\int_\M qa_1\overline{a_2}dV_g(x)dt}+C\norm{A}_{L^\infty((0,T)\times\M_1;T^\ast\M_1)},\]and, by combining the above with (\ref{6.3}), we deduce that\begin{equation}\begin{split}\label{6.4}\abs{\int_0^T\int_{\d_+S_y\M_1}I_0[q(t,\cdot)](y,\theta)\chi^2(t)h(t,y,\theta)}d\theta dt\leq&C\Big(\lambda\tau^{-6}\norm{A}_{L^\infty((0,T)\times\M;T^\ast\M)}\norm{h}_{H^4((0,T)\times\d_+S_y\M)}\\&+\gamma\tau^{-8}\lambda^{6}\norm{h}_{H^5((0,T)\times\d_+S_y\M_1)}+\lambda^{-1}\tau^{-4}\norm{h}_{H^4((0,T)\times\d_+S_y\M_1)}\Big).\end{split}\end{equation}
By the Sobolev interpolation theorem, we can choose $p\in(n+1,\infty)$ such that $\norm{A}_{L^\infty((0,T)\times\M_1;T^\ast\M_1)}\leq C\norm{A}_{W^{1,p}((0,T)\times\M_1;T^\ast\M_1)}$, and by interpolation together with condition (\ref{C1}) we deduce that\[\begin{split}\norm{A}_{L^\infty((0,T)\times\M_1;T^\ast\M_1)}\leq C\norm{A}_{W^{2,p}((0,T)\times\M_1;T^\ast\M_1)}^{\frac{1}{2}}\norm{A}_{L^p((0,T)\times\M_1;T^\ast\M_1)}^{\frac{1}{2}}\\\leq \norm{A}_{L^p((0,T)\times\M_1;T^\ast\M_1)}^{\frac{1}{2}}\leq C\norm{A}_{L^2((0,T)\times\M_1;T^\ast\M_1)}^{\frac{1}{p}}.\end{split}\]
By combining this estimate with the result Theorem \ref{T2}, we conclude that\[\norm{A}_{L^\infty((0,T)\times\M_1;T^\ast\M_1)}\leq C\gamma^{\frac{s_1}{p}}.\]
Thus, we can rewrite (\ref{6.4}) as\begin{equation}\begin{split}\label{6.5}\abs{\int_0^T\int_{\d_+S_y\M_1}I_0[q(t,\cdot)](y,\theta)\chi^2(t)h(t,y,\theta)}d\theta dt\leq C\Big(\lambda\tau^{-6}\gamma^{\frac{s_1}{p}}\norm{h}_{H^4((0,T)\times\d_+S_y\M)}\\+\gamma\tau^{-8}\lambda^{6}\norm{h}_{H^5((0,T)\times\d_+S_y\M_1)}+\lambda^{-1}\tau^{-4}\norm{h}_{H^4((0,T)\times\d_+S_y\M_1)}\Big).\end{split}\end{equation}

\begin{proof}[Proof of Theorem \ref{T3}]
In order to prove (\ref{1.6}) we will use the estimate (\ref{6.5}) together with a suitable choice of $h$. First, note that according to condition (\ref{1.4}) we have $q\in H^5((0,T)\times\M_1)$ with $\textrm{supp}\ q(t,\cdot)\subset\M$ when $t\in(0,T)$. Recall, according to \cite[Section 7]{32}, that $I_0^\ast I_0$ with $I_0^\ast$ the adjoint of $I_0$ (see for instance [2, Subsection 2.2] for details) is an elliptic pseudodifferential operator of order $-1$ for $\xi\in T^\ast\M$. Therefore, for all $t\in(0,T)$, we have $\norm{I_0^\ast I_0[q(t,\cdot)]}\in H^5((0,T)\times\M_1)$ and condition (\ref{1.5}) implies\begin{equation}\label{4.4}\norm{I_0^\ast I_0q}_{H^5((0,T)\times\M_1)}\leq C\norm{q}_{H^5((0,T)\times\M_1)}\leq CB_1.\end{equation}
Moreover, according to \cite[Theorem 4.2.1]{28}, for all $k\in\N$, the operator $I_0:H^k(\M_1)\rightarrow H^k(\d_+S\M_1)$ is bounded. Thus, we can choose $h(t,\cdot)=I_0I_0^\ast I_0[q(t,\cdot)]\in H^5((0,T)\times\d_+S\M_1)$. Integrating the left hand side of (\ref{6.5}) with respect to $y\in\d\M_1$ and applying Fubini's theorem yields\[\int_0^T\chi^2(t)\int_{\d_+S\M_1}I_0[q(t,\cdot)](y,\theta)h(t,y,\theta)\abs{\pair{\theta,\nu(y)}_g}d\theta d\sigma_g(y)dt=\int_0^T\chi^2(t)\int_{\M_1}\abs{I_0^\ast I_0[q(t,\cdot)](x)}^2dV_g(x)dt.\]
Combining this with (\ref{6.5}) and (\ref{4.4}), and using the fact that $\M_1$ is compact, we get\begin{equation}\label{4.5}\int_0^T\chi^2(t)\int_{\M_1}\abs{I_0^\ast I_0[q(t,\cdot)](x)}^2dV_g(x)dt\leq C\Big(\gamma^{\frac{s_1}{p}}\lambda\tau^{-6}+\gamma\tau^{-8}\lambda^{-6}+\tau^{-4}\lambda^{-1}\Big),\end{equation}with $C$ depending only on $\M_1$, $T$ and $B_1$. Further, by the same argument as in (\ref{b}), the estimate (\ref{4.5}) can be rewritten as\begin{equation}\begin{split}\label{4.6}\int_0^T\int_{M_1}\abs{I_0^\ast I_0[q(t,\cdot)](x)}^2dV_g(x)\leq C\Big[\gamma^{\frac{s_1}{p}}\lambda\tau^{-6}+\gamma\tau^{-8}\lambda^{6}+\tau^{-4}\lambda^{-1}+\tau\Big].\end{split}\end{equation}
Note that for all $t\in(0,T)$ we have $\textrm{supp}\ q(t,\cdot)\subset\M$. Thus, according to \cite[Theorem 3]{32}, we have\[\int_{\M_1}\abs{q(t,x)}^2dV_g(x)\leq C\norm{I_0^\ast I_0[q(t,\cdot)]}_{H^1(\M_1)}^2,\quad t\in(0,T).\]
Integrating with respect to $t\in(0,T)$ yields\[\int_0^T\int_{\M_1}\abs{q(t,x)}^2dV_g(x)\leq C\norm{I_0^\ast I_0[q(t,\cdot)]}_{L^2(0,T;H^1(\M_1))}^2.\]
Then, by interpolation we obtain\[\begin{split}\int_0^T\int_{\M_1}\abs{q(t,x)}^2dV_g(x)\leq& C\norm{I_0^\ast I_0[q(t,\cdot)]}_{L^2((0,T)\times\M_1)}\norm{I_0^\ast I_0[q(t,\cdot)]}_{L^2(0,T;H^2(\M_1))}\\\leq& C\norm{I_0^\ast I_0[q(t,\cdot)]}_{L^2((0,T)\times\M_1)},\end{split}\]where $C$ depends on $\M$, $T$ and $B_1$. Combining this with estimate (\ref{4.6}), we find that\begin{equation}\label{4.7}\int_0^T\int_{\M_1}\abs{q(t,x)}^2dV_g(x)\leq C\Big[\gamma^{\frac{s_1}{p}}\lambda\tau^{-6}+\gamma\tau^{-8}\lambda^{6}+\tau^{-4}\lambda^{-1}+\tau\Big]\end{equation}and (\ref{1.6}) follows from (\ref{4.7}) by a similar argument to the one used to prove Theorem \ref{T2} from (\ref{edit1}).
\end{proof}

\section*{Acknowledgments}
 The work of YK was partially supported by the Agence Nationale de la Recherche under grant ANR-17-CE40-0029. 
AT was supported by EPSRC DTP studentship EP/N509577/1.
\begin{small}

\end{small}

\end{document}